\documentclass[11pt,a4paper]{article}
\usepackage{amsmath}
\usepackage{amssymb}
\usepackage{amsthm}
\usepackage{latexsym}
\usepackage{amsbsy}
\usepackage[all]{xypic}
\usepackage{amscd}

\newtheorem{thm}{Theorem}[section]

\newtheorem{lem}[thm]{Lemma}

\newtheorem{prop}[thm]{Proposition}
\newtheorem{defn}[thm]{Definition}
\newtheorem{rem}[thm]{Remark}

\begin{document}

\begin{center}
{\Large \bf Cluster-tilted algebras and their intermediate coverings
}
\bigskip

{\large
 Bin Zhu\footnote{Supported by the NSF of China (Grants 10771112) }}
\bigskip

{\small
 Department of Mathematical Sciences,
 Tsinghua University,
   100084 Beijing, P. R. China

 {\footnotesize E-mail: bzhu@math.tsinghua.edu.cn} }
\bigskip


\end{center}

\def\s{\stackrel}
\def\gama{\gamma}
\def\Longrightarrow{{\longrightarrow}}
\def\P{{\cal P}}
\def\A{{\cal A}}
\def\F{\mathcal{F}}
\def\X{\mathcal{X}}
\def\T{\mathcal{T}}
\def\m{\textbf{ M}}
\def\t{{\tau }}
\def\b{\textbf{d}}
\def\K{{\cal K}}

\def\G{{\Gamma}}
\def\e{\mbox{exp}}

\def\righta{\rightarrow}

\def\s{\stackrel}

\def\ncong{\not\cong}

\def\mathbb{\NN}

\def\Hom{\mbox{Hom}}
\def\Ext{\mbox{Ext}}
\def\ind{\mbox{ind}}
\def\coprod{\amalg }
\def\L{\Lambda}
\def\c{\circ}
\def\mu{\multiput}

\renewcommand{\mod}{\operatorname{mod}\nolimits}
\newcommand{\add}{\operatorname{add}\nolimits}
\newcommand{\Rad}{\operatorname{Rad}\nolimits}
\newcommand{\RHom}{\operatorname{RHom}\nolimits}
\newcommand{\uHom}{\operatorname{\underline{Hom}}\nolimits}
\newcommand{\End}{\operatorname{End}\nolimits}
\renewcommand{\Im}{\operatorname{Im}\nolimits}
\newcommand{\Ker}{\operatorname{Ker}\nolimits}
\newcommand{\Coker}{\operatorname{Coker}\nolimits}
\renewcommand{\r}{\operatorname{\underline{r}}\nolimits}
\def \text{\mbox}

\hyphenation{ap-pro-xi-ma-tion}

\begin{abstract} We construct the intermediate coverings of
cluster-tilted algebras by defining the generalized cluster
categories. These generalized cluster categories are Calabi-Yau
triangulated categories with fraction CY-dimension and have also
cluster tilting objects (subcategories). Furthermore we study the
representations of these intermediate coverings of cluster-tilted
algebras.

\end{abstract}


\medskip


\medskip

\section{Introduction}

 Cluster-tilted algebras are by definition, the endomorphism algebras of cluster tilting objects
 in the cluster categories of hereditary algebras. They together with cluster categories
provide an algebraic understanding of combinatorics of cluster
algebras defined and studied by Fomin and Zelevinsky in [FZ].
 In this connection, the indecomposable exceptional objects in cluster categories
 correspond to the cluster variables, and cluster tilting objects(= maximal $1-$orthogonal subcategories [I1, I2])
 to clusters of corresponding cluster algebras, see
[CK1, CK2].  Cluster categories are the orbit
 categories $\frac{D^{b}(H)}{<F>}$ of derived categories $D^{b}(H)$ of
 a hereditary algebra $H$ by an automorphism group generated by $F=\tau ^{-1}[1]$ where $\tau$
 is the Auslander-Reiten translation in $D^{b}(H)$, $[1]$ is the shift functor of
 $D^{b}(H).$ They are triangulated categories  and are Calabi-Yau categories of CY-dimension $2$ [K].
  Moreover by [KR1] or [KZ], cluster-tilted algebras provide a class of Gorenstein
  algebras of Gorenstein dimension $1$, which are important in
  representation theory of algebras [Rin2].

\medskip

Now let $\mathcal{H}$ be a hereditary abelian category with tilting
objects. The endomorphism algebra of a tilting object in
$\mathcal{H}$ is called quasi-tilted algebra, which consists of
tilted algebras and canonical algebras [H2]. From [H2],
$\mathcal{H}$ is  either derived equivalent to mod$ H $ of a
hereditary algebra $H$  or to the category $coh P$  of coherent
sheaves over a weighted projective line $P$. The later is derived
equivalent to the module categories of canonical algebras [Rin1].
From such a hereditary abelian category $\mathcal{H}$, one can also
define cluster category $\mathcal{C(H)}$ as the orbit category of
$D^b(\mathcal{H})$ by $\tau^{-1}[1]$. It shares the same tilting
theory as the classic case by [BMRRT][Zh]. It was shown that any
cluster tilting object is induced from a tilting object of a
hereditary abelian category which is derived equivalent to
$\mathcal{H}$.  The endomorphism algebra of a tilting object in
$\mathcal{C(H)}$ is now
 called a cluster-tilted algebra of type $\mathcal{H}$.
\medskip

  The aim is to study the
cluster-tilted algebras in this general setting: Firstly, for any
positive integer $m$, we associate a generalized cluster category
which is defined as the orbit categories of derived categories
$D^b(\mathcal{H})$ by the group $<F^m>$ generated by $F^m$. It is a
triangulated category by Keller [K], which are Calabi-Yau categories
of Calabi-Yau dimension $2m/m$. The cluster tilting objects in this
generalized cluster category are shown to correspond one-to-one to
ones in the classical cluster categories; the endomorphism algebras
of cluster tilting objects in $D^b(\mathcal{H})/<F^m>$ are the
coverings of the cluster-tilted algebras. They all share an
universal covering: the endomorphism algebra of corresponding
cluster tilting subcategory in $D^b(\mathcal{H})$.
 Secondly, we study the problem when the endomorphism algebras of rigid objects are cluster-tilted algebras.
  It was pointed by Buan, Marsh and Reiten in [BMR2] that there are examples of rigid objects in cluster categories
  whose endomorphism algebras are not again
  cluster-tilted.  We generalize the Assem-Bruestle-Schiffler's  characterization [ABS1] of cluster-tilted algebras
  to cluster-tilted algebras of type $\mathcal{H}$, i.e., we show
that the trivial extension algebra $A=B\ltimes Ext^2_B(DB,B)$ is a
cluster-tilted algebra of type $\mathcal{H}$ if and only if $B$ is a
quasi-tilted algebra. Using this characterization, we prove that the
endomorphism algebra of certain  rigid object (triangular rigid, see
Section 3 for details) is a cluster-tilted
 algebra.

The article is organized as follows:

In Section 2 we collect basic material on cluster tilting
subcategories (objects) in triangulated categories.

Section 3 contains a generalization of Assem-Bruestle-Schiffler's
characterization [ABS1] of cluster-tilted algebras to the general
case: the cluster-tilted algebras of type $\mathcal{H}$, i.e., we
show that the trivial extension algebra $A=B\ltimes Ext^2_B(DB,B)$
is a cluster-tilted algebra of type $\mathcal{H}$ if and only if $B$
is a quasi-tilted algebra.
 The main result is that
 the endomorphism algebra of a triangular rigid object (see
Section 3 for details) is a cluster-tilted algebra.

In Section 4 we first introduce the generalized cluster categories
$\mathcal{C}_{F^m}(\mathcal{H})$, which are triangulated categories
and are coverings of the corresponding cluster categories
$\mathcal{C(H)}$. We then study cluster tilting theory in these
triangulated categories. We show that the covering functors $\pi_m:
D^b(\mathcal{H})\rightarrow \mathcal{C}_{F^m}(\mathcal{H})$ and
$\rho_m: \mathcal{C}_{F^m}(\mathcal{H})\rightarrow \mathcal{C(H)}$
induce covering functors from the subcategory of projective modules
of the endomorphism algebra of cluster tilting subcategory in
$D^b(\mathcal{H})$ to the subcategory of projective modules of the
generalized cluster-tilted algebra of a tilting object in
$\mathcal{C}_{F^m}(\mathcal{H})$, respectively, from the subcategory
of projective modules of the generalized cluster-tilted algebra of a
tilting object in $\mathcal{C}_{F^m}(\mathcal{H})$  to the
cluster-tilted algebra of the corresponding cluster tilting objects
in $\mathcal{C(H)}$; and that it also gives the corresponding
push-down functors between their module categories.

 \medskip

\bigskip

\section{Basics on cluster tilting subcategories}

Let $\mathcal{D}$ be a $k-$linear triangulated category with finite
dimensional Hom-spaces over a field $k$ and with Serre duality. We
assume that
 $\mathcal{D}$ is a Krull-Schmidt-Remark category. Let  $\mathcal{T }$ be a full
subcategory of $\mathcal{D}$ closed under taking direct summands.
The quotient category of $\mathcal{D}$ by $\mathcal{T }$ is denoted
by  $\mathcal{D/T }$, is by definition, a category with the same
objects as $\mathcal{D}$ and the space of morphisms from $X$ to $Y$
is the quotient of group of morphisms from $X$ to $Y$ in
$\mathcal{D}$ by the subgroup consisting of morphisms factor through
an object in $\mathcal{T}$. The quotient $\mathcal{D/T }$ is also an
additive Krull-Schemidt category. For $X, Y\in \mathcal{D}$, we use
$Hom(X, Y)$ to denote $Hom_{\mathcal{D}}(X,Y)$ for simplicity, and
define that $Ext^k(X,Y):=Hom(X,Y[k])$. For a subcategory
$\mathcal{T}$, we say that $Ext^i(\mathcal{T},\mathcal{T})=0$
provided that $Ext^i(X,Y)=0$ for any $X,Y\in \mathcal{T}$. For an
object $T$, add$T$ denotes the full subcategory consisting of direct
summands of direct sum of finitely many copies of $T$. Throughout
the article, the composition of morphisms $f: M\rightarrow N$ and
$g:N\rightarrow L$ is denoted by $fg: M\rightarrow L.$

  \medskip
Fix a triangulated category $\mathcal{D}$, and assume that
$\mathcal{T}$ is a functorially finite subcategory of $\mathcal{D}$.
\medskip

\begin{defn} \begin{enumerate}
\item  $\mathcal{T}$ is called rigid, provided $Ext^1(\mathcal{T},\mathcal{T})=0$; in particular,
 an object $T$ is called rigid provided $Ext^1(T,T)=0$.

\item  $\mathcal{T}$ is called a cluster tilting subcategory provided $X\in \mathcal{T}$ iff
 $Ext^1(X, \mathcal{T})=0$  and  $X\in \mathcal{T}$ iff $Ext^1( \mathcal{T},X)=0$.
 An object $T$ is a cluster tilting object if and
only if add$T$ is a cluster tilting subcategory.

\end{enumerate}

\end{defn}

\begin{rem}
\begin{enumerate}
\item Not all triangulated categories have cluster tilting subcategories, see the fillowing example.

 \textbf{Example} Let $A=kQ/I$ be the self-injective algebra given
by the quiver $Q$

\vspace*{-1cm}
\begin{center}
 \setlength{\unitlength}{0.61cm}
 \begin{picture}(5,4)
 \put(0,2){a}\put(0.4,2.2){$\circ$}
\put(3,2.2){$\circ$}\put(3.4,2){b} \put(0.8,2.5){\vector(3,0){2}}
\put(2.8,2.2){\vector(-1,0){2}}
 \put(1.7,2.7){$\alpha$}\put(1.7,1.5){$\beta$}
 \end{picture}
 \end{center}
\vspace*{-1cm} and the relations $\alpha \beta\alpha, \  \beta\alpha
\beta $.

\newcommand{\Maba}{\begin{array}{c} a \\ b \\ a \end{array}}
\newcommand{\Mbab}{\begin{array}{c} b \\ a \\ b \end{array}}
\newcommand{\Mab}{\begin{array}{c} a \\ b \end{array}}
\newcommand{\Mba}{\begin{array}{c} b \\ a \end{array}}

The Auslander Reiten quiver of $A-mod$ looks as follows:
$$\begin{array}{ccccccccc}
\Mbab &&&& \Maba &&&& \Mbab \\
& \searrow && \nearrow && \searrow && \nearrow & \\
&& \Mba &&&& \Mab && \\
& \nearrow && \searrow && \nearrow && \searrow & \\
a &&&& b &&&& a
\end{array}$$
Here, the first and the last column are identified.

Deleting the first row produces the Auslander Reiten quiver of the
stable category $A-\underline{mod}$:

$$\begin{array}{ccccccccc}
&& \Mba &&&& \Mab && \\
& \nearrow && \searrow && \nearrow && \searrow & \\
a &&&& b &&&& a
\end{array}$$

This stable category ${\mathcal H} = A-\underline{mod}$ of $A$ has
no cluster tilting objects. We note that any indecomposable objects
is a maximal rigid.

\item In a module category $\Lambda-mod$ of a self-injective algebra $\Lambda$, $T\oplus
\Lambda$ is a cluster tilting module (=maximal $2-$orthogonal module
in [I1, I2]) if and only if $T$ is cluster tilting in
$\Lambda-\underline{mod}$.
\end{enumerate}
\end{rem}

\begin{rem} It was proved in [KZ] that if $\mathcal{T}$ is
contravariantly finite and satisfies the condition that $X\in
\mathcal{T}$ if and only if $Ext^1( \mathcal{T},X)=0$, then
$\mathcal{T}$ is a cluster tilting subcategory.\end{rem}
\medskip

For a triangulated category $\mathcal{D}$ with Serre duality
$\Sigma$, $\mathcal{D}$ has Auslander-Reiten triangles $\tau$ and
$\Sigma=\tau [1],$  where $\tau$ is the Auslander-Reiten
translation. Denote by $F=\tau^{-1}[1].$

\begin{lem} Let $\mathcal{D}$ be a triangulated category with Serre duality $\Sigma$, and
$\mathcal{T }$ a cluster tilting subcategory of $\mathcal{D}$. Then
$F\mathcal{T}=\mathcal{T}$.\end{lem}

\medskip
\begin{proof} The assertion was proved in [KZ] or [IY].

\end{proof}

\medskip

The following results were proved in [KZ]

\begin{thm} Let $T$ be a cluster tilting object of a
triangulated category $\mathcal{D}$, and $A=End_{\mathcal{D}}T$.
Then the followings hold:

\begin{enumerate}

\item{} The functor $Hom(T,-):\mathcal{D}\rightarrow modA$ induces a equivalence $\mathcal{D}/add(T[1])\cong A-mod $,
 and $A$ is a Gorenstein algebra of Gorenstein dimension at most $1$.

\item{} Assume that the field $k$ is algebraically closed.
If $B=End_{\mathcal{D}}T'$ is the endomorphism algebra of  another
cluster tilting object $T'$, then $A$ and $B$
 have same representation type.

\end{enumerate}

\end{thm}

\medskip

 Let $T=T_1\oplus T'$ be a cluster tilting object of a
triangulated category $\mathcal{D}$, where $T_1$ is indecomposable
object. Let $T_1^*\rightarrow E\s{f}{\rightarrow}T_1 \rightarrow
T_1^*[1]$ be the triangle with $f$ the minimal right
$addT'-$approximation of $T_1$. It follows from [IY] that
$T^*=T_1^*\oplus T'$ is a cluster tilting object and there is a
triangle $T_1\rightarrow E'\s{g}{\rightarrow}T_1^* \rightarrow
T_1[1]$  with $g$ being the minimal right $addT'-$approximation of
$T_1^*$. Let $A$, $B$ be the endomorphism algebras of cluster
tilting objects $T$, $T^*$ respectively. Denote by $S_{T_1}$, (or
$S_{T_1^*}$) the simple $A-$module corresponding to $T_1$ (resp.
simple $B-$module corresponding to $T_1^*$). The following
proposition is a generalization of Proposition 2.2 in [KR1].

\medskip

\begin{prop} Let $T$ and $T^*$ be as above. Then $modA/ add S_{T_1}\approx mod B/add S_{T_1^*}.$

\end{prop}
\begin{proof} Denote by $G=\mbox{Hom}(T,-).$ The
induced functor $\bar{G}: \mathcal{D)}/\mbox{add}(T[1]) \rightarrow
A-\mbox{mod}$ is an equivalence by Theorem 2.5. We consider the
composition of functor $\bar{G}$ with the quotient functor $Q:
A-\mbox{mod} \rightarrow \frac{\Lambda
-\mbox{mod}}{\mbox{add}(\mbox{Hom}(T, T_1^*[1]))}$, which is denoted
by $G_1$. The functor $G_1$ is full and dense since $\bar{G}$ and
$Q$ are. Under the equivalence $\bar{G}$, $T_1^*[1]$ corresponds to
Hom$(T,  T_1^*[1]).$  For any morphism $f: X\rightarrow Y$ in the
category $\frac{\mathcal{D}}{\mbox{add}(T[1])}$, $\bar{G}(f):
G(X)\rightarrow G(Y)$ factors through add$(\mbox{Hom}(T,T_1^*[1]))$
if and only if $f$ factors through add$T_1^*[1].$ Then $G_1$ induces
an equivalence, denoted by $\bar{G_1},$  from the category
$\frac{\cal{D}}{\mbox{add}(T[1]\oplus T_1^*[1]))}$ to  the category
$\frac{A -\mbox{mod}}{\mbox{add}(\mbox{Hom}(T,\tau T_1^*[1]))}.$
 Therefore we have that $\frac{A-\mbox{mod}}{\mbox{add(Hom}(T, T^*_1[1]))} \approx
\frac{B-\mbox{mod}}{\mbox{add(Hom}(T',T_1[1]))}.$ It is easy to
prove that $\mbox{Hom}(T, T_1^*[1]))\cong S_{T_1}$ and Hom$(T^*,
T_1[1])\cong S_{T_1^*}$ (compare Lemma 4.1 in [BMR1]). Then
 $modA/ add S_{T_1}\approx B/add S_{T_1^*}.$
The proof is finished.

\end{proof}
\section{Cluster-tilted algebras of type $\mathcal{H}$}

In this section, $\mathcal{H}$ will denote a hereditary $k-$linear
category with finite dimensional Hom-spaces and Ext-spaces. We also
assume that $\mathcal{H}$ has tilting objects. The endomorphism
algebra of  tilting object $T$ in $\mathcal{H}$ is called a
quasi-tilted algebra. Since $\mathcal{H}$ has tilting objects,
$D^b(\mathcal{H})$ has Serre duality, and has also Auslander-Reiten
triangles, the Auslander-Reiten translation is denoted by $\tau$.
Let $F=\tau^{-1}[1]$ be the automorphism of the bounded derived
category $D^b(\mathcal{H})$. We call the orbit category
$D^b(\mathcal{H})/<F>$ the cluster category of type $\mathcal{H},$
which is denoted by $\mathcal{C(H)}$ [BMRRT], [CCS1]. For cluster
tilting theory
 in the cluster category $\mathcal{C(H)}$, we refer [BMRRT, Zh]. The
endomorphism algebra $End_{\mathcal{C(H)}}T$ of a cluster tilting
object $T$ in $\mathcal{C(H)}$ is called a cluster-tilted algebra of
type $\mathcal{H}$. When $\mathcal{H}$ is the module category over a
hereditary algebra $H=kQ$, we call the corresponding orbit category
the cluster category of $H$ or of $Q$. In this case the endomorphism
algebra of a cluster tilting object is called a cluster-tilted
algebras of $H$, or simply of $Q$ [BM, BMR, Zh, ABS1, ABS2].

Now we give a characterization of cluster-tilted algebras of type
$\mathcal{H}$, which generalizes some results in [ABS1], [Zh].
\medskip

Given any finite-dimensional algebra $B$, from the $B-$bimodule
Ext$^2(DB,B)$, one can form the trivial extension algebra of $B$
with the bimodule Ext$^2(DB,B)$: $A=B\ltimes \mbox{Ext}^2(DB,B)$. It
was proved that this trivial extension algebra is a cluster-tilted
algebra of $H$ if and only if $B$ is a tilted algebra [ABS1], which
can be viewed as a completion of the description of cluster-tilted
algebras given in [Zh]. In the following, we generalize the
characterization of cluster-tilted algebras to the cluster-tilted
algebras of type $\mathcal{H}$.

\begin{prop} Let $A=B\ltimes
\mbox{Ext}^2(DB,B)$. Then $A$ is a cluster-tilted algebra of type
$\mathcal{H}$ for some hereditary abelian category
 $\mathcal{H}$ if and only if $B$ is a quasi-tilted algebra, i.e.
 the endomorphism algebra of a tilting object in $\mathcal{H}$.

\end{prop}

\begin{proof} Suppose that $A$ is a cluster-tilted algebra of type $\mathcal{H}$. It follows from
 [Zh] that $A=B\ltimes
D\mbox{Hom}_{D^b(\mathcal{H})}(T,\tau^{-1}T[1])$, where
$B=\mbox{End}_{\mathcal{H}}T$
 is the quasi-tilted algebra. It remains to prove that the bimodule Hom$_{D^b(\mathcal{H})}(T,\tau^{-1}T[1])$ is isomorphic to
$Ext^2_B(DB,B)$:
 $$\begin{array}{ll}Hom_{D^b(\mathcal{H})}(T,\tau ^{-1}T[1])&\cong Hom_{D^b(B)}(B,\tau ^{-1}B[1])\\
&\cong Hom_{D^b(B)}(\tau B [1],B[2])\\
 &\cong Hom_{D^b(B)}(DB,B[2])\\
 &\cong Ext^2_{B}(DB, B).\end{array}$$

 The proof for the other direction: Suppose that  $A=B\ltimes
D\mbox{Ext}^2_B(DB,B)$ with $B$ being a
 quasi-tilted algebra, i.e. $B=\mbox{End}T$ of tilting object $T$ in
 a hereditary abelian category $\mathcal{H}$. Hence $T$ is a cluster
 tilting object in the cluster category $\mathcal{C(H)}$ of type $\mathcal{H}$ [Zh]. Then
 the endomorphism algebra $\mbox{End}_{\mathcal{C(H)}}T$ is
 isomorphic to $\tilde{B}$. Therefore $A$ is a cluster-tilted algebra of type $\mathcal{H}$.
\end{proof}

Applying the characterization of cluster-tilted algebras to the
endomorphism algebras of rigid objects in cluster categories, we
prove that the endomorphism algebras of certain rigid objects are
also cluster-tilted. In general it is not true as pointed by
Buan-Marsh-Reiten in [BMR2].

We call that a rigid object $T$ in the cluster category
$\mathcal{C(H)}$ is triangular provided $T$ is an exceptional object
in a hereditary category $\mathcal{H}'$ derived to $\mathcal{H}$ and
there is a complement  $T'\in\mathcal{H} $ such that $T\oplus T'$ is
a tilting object in $\mathcal{H}'$, and  satisfies that
$Hom_{\mathcal{H}}(T,T')=0$ or $Hom_{\mathcal{H}}(T',T)=0$.
\medskip

\begin{thm} Let $T$ be a triangular rigid object in the cluster category $\mathcal{C(H)}$.
Then $End_{\mathcal{C(H)}}T$
is a cluster-tilted algebra.

\end{thm}

\begin{proof} Since $T$ is a triangular rigid object, there exists an
object $T'$ such that $T\oplus T'$ is a tilting object in a
hereditary category $\mathcal{H}'$ which is derived equivalent to
$\mathcal{H}$ and  $Hom_{\mathcal{H}'}(T',T)=0$, or
$Hom_{\mathcal{H}'}(T,T')=0$. We assume that
$Hom_{\mathcal{H}'}(T',T)=0$, the other case is exactly similar. Let
 $G=Hom_{D^b(\mathcal{H}')}(T,-)$ be the derived functor of
 $Hom_{\mathcal{H}'}(T,-)$. It is a faithful and full functor from
 $D^b(\mathcal{H}')$ to $D^b(End_{\mathcal{H}'}T)$ [H1].
 Since $T'$ is exceptional in $\mathcal{H}'$, the right perpendicular category $(T')^{\perp}:=\{ X\in \mathcal{H}'\ |\
 Hom_{\mathcal{H}'}(T', X)=0=Ext^1_{\mathcal{H}'}(T,X)\ \}$
 is a full and extension-closed subcategory of $\mathcal{H}'$, which
 is also hereditary [GL]. It is easy to see that $T$ is a
 tilting object in  $(T')^{\perp}$. Hence the endomorphism algebra
  $B=End_{\mathcal{H}'}T=End_{(T')^{\perp}}T$ is quasi-tilted. Now $EndT\cong
 Hom_{D^b(\mathcal{H}')}(T,T)\oplus
 Hom_{D^b(\mathcal{H}')}(T,\tau^{-1}T[1])\cong Hom_{D^b(\mathcal{H}')}(T,T)\oplus
 Hom_{D^b(\mathcal{H}')}(\tau T[1],T[2])\cong Hom_{D^b(\mathcal{H}')}(T,T)\oplus
 Hom_{D^b(B)}(G(\tau T[1]),G(T[2]))\cong B \oplus
 Hom_{D^b(B)}(DB,B[2])\cong B \oplus
 Ext^2_B(DB,B).$ By Theorem 3.1, $End_{\mathcal{C(H)}}T$ is a cluster-tilted algebra.

\end{proof}

\begin{rem} Combining corollary above with Proposition
2.1(d) in [KR1] or Theorem 4.3 in [KZ] which assert that any cluster
tilted algebra is Gorenstein algebra of Gorenstein dimension $1$, we
 have that $EndT$ is a Gorenstein algebra of Gorenstein
dimension $1$ for any triangular exceptional object $T$ in cluster
category $\mathcal{C(H)}$.

\end{rem}

\medskip

\textbf{Example} Let $\mathcal{C}(kQ)$ be the cluster category of
type $Q$, where $Q$ is the quiver :\begin{center}
 \setlength{\unitlength}{0.61cm}
 \begin{picture}(7,4)
 \put(0,2){a}\put(0.4,2.2){$\circ$}
\put(3,2.2){$\circ$}\put(3.4,2){b}
\put(6,2.2){$\circ$}\put(6.4,2){c}
\put(9,2.2){$\circ$}\put(9.4,2){d} \put(0.8,2.5){\vector(3,0){2}}

\put(3.8,2.5){\vector(3,0){2}}
 \put(6.8,2.5){\vector(3,0){2}}

 \end{picture}
 \end{center}

 If we take $T'=P_a\oplus \tau ^{-3}P_d\oplus P_c$,  $T'$ is a  triangular rigid object
  in the cluster category, since $T= P_d\oplus T'$ is a tilting $kQ-$module and $Hom_{\mathcal{C}(kQ)}(T',P_d)=0$.
   Then by Theorem 3.2, the endomorphism algebra of $T'$ is a
   cluster-tiled algebra. In fact, $End_{\mathcal{C}(H)}T'$ is the
   following algebra given by $Q'$:

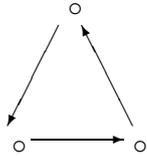
\begin{center}
 \setlength{\unitlength}{0.61cm}
 \begin{picture}(6,6)
 \put(0.4,2.2){$\circ$}
\put(3,2.2){$\circ$}
 \put(1.6, 5.2){$\circ$}

\put(0.8,2.5){\vector(3,0){2}}

\put(3,2.8){\vector(-1,2){1.1}}

 \put(1.4,5){\vector(-1,-2){1.1}}

 \end{picture}
 \end{center}
 with $rad^2=0$.

\section{Intermediate covers of cluster tilted algebras of type $\mathcal{H}$}

 As the previous section, $\mathcal{H}$ denotes a hereditary $k-$linear category with
finite dimensional Hom-spaces and Ext-spaces. We assume that
$\mathcal{H}$ has tilting objects. Since $\mathcal{H}$ has tilting
objects, $D^b(\mathcal{H})$ has Serre duality, and also
Auslander-Reiten translate $\tau$ (AR-translate for short). Let
$F=\tau^{-1}[1]$ be the automorphism of the bounded derived category
$D^b(\mathcal{H})$. Fixed a positive
 integer $m$ throughout this section.

 We consider the orbit category $D^b(\mathcal{H})/<F^m>$, which is
 by definition, a k-linear category whose objects are the same in
 $D^b(\mathcal{H})$, and whose morphisms are given by:
$$\Hom_{D^b(\mathcal{H})/<F^m>}(\widetilde{X},\widetilde{Y}) =
\oplus_{i \in \mathbf{Z}}
 \Hom_{D^b(\mathcal{H}}(X,(F^m)^iY).$$
Here $X$ and $Y$ are objects in $D^b(\mathcal{H})$, and
$\widetilde{X}$ and $\widetilde{Y}$ are the corresponding objects in
$D^b(\mathcal{H})/<F^m>$ (although we shall sometimes write such
objects simply as $X$ and $Y$).

\begin{defn} The orbit category $D^b(\mathcal{H})/<F^m>$  is called the generalized cluster category of
type $\mathcal{H}$. We denote it by
$\mathcal{C}_{F^m}(\mathcal{H})$.\end{defn}

\begin{rem} When $m=1$, we get back to the usual cluster category $\mathcal{C(H)}$,
which were introduced by Buan-Marsh-Reineke-Reiten-Todorov in
[BMRRT], and also by Caldero-Chapoton-Schiffler in [CCS] for $A_n$
case. \end{rem}

The generalized cluster categories $\mathcal{C}_{F^m}(\mathcal{H})$
 serve as intermediate categories  between the corresponding
cluster categories  $\mathcal{C(H)}$ and derived categories
$D^b(\mathcal{H})$. Similar as for the case of cluster categories,
for any positive integer $m$, We have a natural projection functor
$\pi_{m}: D^b(\mathcal{H})\rightarrow
\mathcal{C}_{F^m}(\mathcal{H})$. If $m=1$, the projection functor
$\pi_{m}$ is simply denoted by $\pi$.

  Now we define a functor  $\rho _{m}: \mathcal{C}_{F^m}(\mathcal{H})\longrightarrow
\mathcal{C}(\mathcal{H})$, which sends objects $\tilde{X}$ in
$\mathcal{C}_{F^m}(\mathcal{H})$ to objects $\tilde{X}$ in
$\mathcal{C}(\mathcal{H})$  and  morphisms $f: \tilde{X}\rightarrow
\tilde{Y}$ in $\mathcal{C}_{F^m}(\mathcal{H})$ to the morphisms $f:
\tilde{X}\rightarrow \tilde{Y}$ in $\mathcal{C}(\mathcal{H})$.

It is easy to check that $\pi=\rho_{m}\circ\pi_{m}.$

One can identify the set ind$\mathcal{C(H)}$ with the fundamental
domain for the action of $F$ on ind$D^b(\mathcal{H})$ [BMRRT].
 Passing to the orbit category $\mathcal{C}_{F^m}(\mathcal{H})$, one
can view ind$\mathcal{C(H)}$ as a (probably not fully) subcategory
of ind$\mathcal{C}_{F^m}(\mathcal{H})$.

\medskip

\begin{prop}\label{pr}\begin{enumerate}

\item $\mathcal{C}_{F^m}(\mathcal{H})$ is a triangulated category with Auslander-Reiten triangles  and Serre functor
$\Sigma =\tau [1]$, where $\tau$ is the AR-translate in
$\mathcal{C}_{F^m}(\mathcal{H})$, which is induced from AR-translate
in $D^b(\mathcal{H})$.
\item The projection $\pi_{m}: D^b(\mathcal{H})\rightarrow
\mathcal{C}_{F^m}(\mathcal{H})$ and $\rho _{m}:
\mathcal{C}_{F^m}(\mathcal{H})\longrightarrow
\mathcal{C}(\mathcal{H})$ are  triangle functors and also covering
functors.
\item   $\mathcal{C}_{F^m}(\mathcal{H})$ is a Calabi-Yau category of CY-dimension $\frac{2m}{m}$.
\item  $\mathcal{C}_{F^m}(\mathcal{H})$ is  a Krull-Remark-Schmidt category.
\item $ \mathrm{ind }\mathcal{C}_{F^m}(\mathcal{H})=\bigcup_{i=0}^{i=m-1}(\mathrm{ind}F^i(\mathcal{C(H)})).$
\end{enumerate}
\end{prop}

\begin{proof}\begin{enumerate}\item It follows from [Ke] that $\mathcal{C}_{F^m}(\mathcal{H})$ is a triangulated
category.  The remains follow from Proposition 1.3 \cite{BMRRT}.
\item It is proved in Corollary 1 in Section 8.4 of \cite{Ke} that $\pi_{m}: D^b(\mathcal{H})\rightarrow
\mathcal{C}_{F^m}(\mathcal{H})$ is a triangle functor. It is easy to
check that $\pi\circ F^m\cong \pi$. By the universal property of the
orbit category  $D^b(\mathcal{H})/<F^m>$ [Ke], [KR2] we obtain a
triangle functor $\rho: \mathcal{C}_{F^m}(\mathcal{H})\rightarrow
\mathcal{C(H)}$ satisfying that $\rho \pi_m=\pi$, which turns out to
be the functor $\rho_m$.
\item The Serre functor $\Sigma=\tau [1]$ in $\mathcal{C}_{F^m}(\mathcal{H})$ satisfies that $\Sigma ^m=\tau^m[m]=F^m[2m]
\cong [2m]$. Therefore $\mathcal{C}_{F^m}(\mathcal{H})$ is a
Calabi-Yau category with CY-dimension $\frac{2m}{m}$.
\item The proof for $m=1$ is given in Proposition 1.6 \cite{BMRRT}, which can be modified for the general $m$.
\end{enumerate}
\end{proof}

  We note that
   if the hereditary abelian category $\mathcal{H}$ is equivalent to the module category of a finite dimensional
   hereditary algebra $H$, then the indecomposable objects in $\mathcal{C(H)}$ are of form $\tilde{M}$, where $M$
   is an indecomposable $H-$module, or of form $\tilde{P[1]}$, where $P[1]$ is the first shift
   of an indecomposable projective $H-$module $P$;  if  the hereditary abelian category $\mathcal{H}$ is not
   equivalent to the module category of a finite dimensional
   hereditary algebra $H$, then the indecomposable objects in $\mathcal{C(H)}$ are of form $\tilde{M}$, where $M$
   is an indecomposable object in $\mathcal{H}$.

   Now we discus the cluster tilting objects in
$\mathcal{C}_{F^m}(\mathcal{H})$. Denoted by $F=\tau^{-}[1]$, which
 can be viewed an automorphism of $D^b(\mathcal{H})$ or of $\mathcal{C}_{F^m}(\mathcal{H})$.
 The following proposition is a generalization of Lemma 4.14 in [KZ].

\begin{prop} An object $T$ in $\mathcal{C}_{F^m}(\mathcal{H})$ is a cluster tilting object if and only if
$\pi_m^{-1}(addT)$ is a cluster tilting subcategory of
$D^b(\mathcal{H})$ \label{tiltinginverse}
\end{prop}

\begin{proof} We divide the proof into two cases: the case when $\mathcal{H}$ is equivalent to
the module category of a finite dimensional
   hereditary algebra $H$, and the case when $\mathcal{H}$ is not
   equivalent to the module category of a finite dimensional
   hereditary algebra. We give the detail proof of the proposition for the first case.
   The proof for the second case
    is similar as the first one, we omit it.

    Suppose that $\mathcal{H}\cong H-mod$, where $H$ is a finite dimensional hereditary algebra over a field $k$.
      For an object $T$ in $\mathcal{C}_{F^m}(\mathcal{H})$, we denote
$\mathcal{T}=\pi_m ^{-1}(addT ),$  which is the full subcategory of
$D^b(H)$. It is easy to prove that $F(\mathcal{T})=\mathcal{T}$ in
$D^b(\mathcal{H})$ if and only if $F(addT)=addT$ in
$\mathcal{C}_{F^m}(\mathcal{H})$.

Suppose $\mathcal{T}$ is a cluster tilting subcategory of $D^b(H)$.
Then $F\mathcal{T}=\mathcal{T}$ by Lemma 2.4 or Proposition 4.7
[KZ]. Hence $F(addT)=addT$ in $\mathcal{C}_{F^m}(\mathcal{H})$. We
denote by $\mathcal{T}'$ the intersection of $\mathcal{T}$ with the
additive subcategory $\mathcal{C}'$ generated by all $H-$modules as
stalk complexes of degree $0$ together with $H[1]$. Then we have
that $\mathcal{T}=\{ F^n(\mathcal{T}')| n\in \mathbf{Z} \}.$ Now
$\pi_m(\mathcal{T})=\pi_m(\bigcup _{i=o}^{i=m-1}F^i(\mathcal{T}'))$,
denoted by $\mathcal{T}_1$. For any pair of objects $\tilde{T}_1,
\tilde{T}_2\in \mathcal{T}_1$, there are $T_1, T_2\in \mathcal{T}'$
such that $\tilde{T}_1 =F^t(\pi_m (T_1)), \tilde{T}_2=F^s(\pi_m
(T_2))$ with $0\le t, s \le m-1.$
 Then  $Ext^1(\tilde{T}_1, \tilde{T}_2)=Hom(\tilde{T}_1, \tilde{T}_2[1])
\cong \oplus _{n\in Z}Hom_{D^b(H)}(F^s(T_1), (F^m)^nF^t(T_2[1])) =
\oplus _{n\in Z}Hom_{D^b(H)}(T_1, F^{mn+t-s}T_2[1]).$ By an easy
computation, one has that \newline
 $Hom_{D^b(H)}(T_1, F^{mn+t-s}T_2[1])=0$
if $nm+t-s\le -2$ or $nm+t-s\ge 1$. When $nm+t-s= -1, $
$Hom_{D^b(H)}(T_1, F^{mn+t-s}T_2[1])=Hom_{D^b(H)}(T_1,
F^{-1}T_2[1])=Hom_{D^b(H)}(T_1, \tau T_2)\cong DExt_{D^b(H)}(T_2,
T_1),$ which equals $0$ by the  fact that
   $\mathcal{T}$ is a cluster tilting subcategory of $D^b(H)$.
When $nm+t-s= 0, $ $Hom_{D^b(H)}(T_1,
F^{mn+t-s}T_2[1])=Hom_{D^b(H)}(T_1, T_2[1])=Ext_{D^b(H)}(T_1, T_2),$
which equals $0$ by the fact that
   $\mathcal{T}$ is a cluster tilting subcategory of $D^b(H)$.
  Therefore  $Ext^1(\tilde{T}_1, \tilde{T}_2)=0$, i.e.
  $\mathcal{T}_1$ is rigid in $\mathcal{C}_{F^m}(\mathcal{H})$.

If there are indecomposable objects $\tilde{X}=\pi_m (X)\in
\mathcal{C}_{F^m}(H)$ with $X\in D^b(\mathcal{H})$ satisfying
$Ext^1(\mathcal{T}_1,\tilde{X})=0$, then $Ext^1(F^n\mathcal{T}',
X)=0$ for any $n$, and then $Ext^1(\mathcal{T}, X)=0$. Hence $X\in
\mathcal{T}$ by $\mathcal{T}$ being a cluster tilting subcategory.
Thus $\tilde{X}\in \mathcal{T}_1.$ This proves that the image
$\mathcal{T}_1$  of $\mathcal{T}$ under $\pi_m$ is a cluster tilting
subcategory of $\mathcal{C}_{F^m}(H).$

Conversely, from $\mathcal{T}=\pi_m^{-1}(\mathcal{T}_1)$ and
$F(\mathcal{T}_1)=\mathcal{T}_1$, we get $F(\mathcal{T})=
\mathcal{T}$. As above we denote by $\mathcal{T}'$ the intersection
of $\mathcal{T}$ with the additive subcategory $\mathcal{C}'$
generated by all $H-$modules as stalk compleses of degree $0$
together with $H[1]$. Then $\mathcal{T}=\{ F^n(\mathcal{T}')| n\in
\mathbf{Z} \}$ and
 $\mathcal{T}_1= \pi_m(\mathcal{T})=\pi_m(\bigcup
_{i=o}^{i=m-1}F^i(\mathcal{T}'))$.  From $\mathcal{T}_1$ being
contravariantly finite, we have $\mathcal{T}$ is also
contravariantly finite.  Since $Ext^1(\mathcal{T}_1, \mathcal{T}_1)
\cong\oplus _{n\in Z} Ext^1_{D^b(H)}(\bigcup
_{i=o}^{i=m-1}F^i(\mathcal{T}'),F^n(\bigcup
_{i=o}^{i=m-1}F^i(\mathcal{T}'))=0$, we have that
$Ext^1_{D^b(H)}(F^m\mathcal{T}',F^n\mathcal{T}')\cong
Ext^1_{D^b(H)}(\mathcal{T}', F^{n-m} \mathcal{T}')=0.$ This proves
that $\mathcal{T}$ is an orthogonal subcategory. Now if $X\in
D^b(H)$ satisfies $Ext^1_{D^b(H)}(\mathcal{T}, X)=0$, then
$Ext^1_{\mathcal{C}_{F^m}(\mathcal{H})}(F^i(\mathcal{T}_1),
\tilde{X})=0, \  \forall 0\le i\le m-1$. It follows that
$\tilde{X}\in \mathcal{T}_1$, hence $X\in \mathcal{T}$. Similarly,
if $X\in D^b(H)$ satisfies $Ext^1_{D^b(H)}(X,\mathcal{T})=0$, then
$X\in \mathcal{T}$.
\end{proof}

From Proposition 4.4 and Lemma 4.14 in [KZ], we have the one-to-one
correspondence between the three sets: the set of cluster tilting
subcategories in $D^b(\mathcal{H})$; the set of cluster tilting
subcategories in $\mathcal{C}_{F^m}(\mathcal{H})$; the set of
cluster tilting subcategories in $\mathcal{C(H)}$, via triangle
covering functors: $\pi_{m}: D^b(\mathcal{H})\rightarrow
\mathcal{C}_{F^m}(\mathcal{H})$ and $\rho _{m}:
\mathcal{C}_{F^m}(\mathcal{H})\longrightarrow
\mathcal{C}(\mathcal{H})$.

\begin{thm} Let $\mathcal{H}$ be a hereditary abelian category with tilting
objects. Let $T\in \mathcal{C(H)}$.
\begin{enumerate}

\item $T$ is a cluster tilting object in cluster category
$\mathcal{C(H)}$ if and only if $\rho_m^{-}(T)$ is a cluster tilting
object in $\mathcal{C}_{F^m}(\mathcal{H})$ if and only if
$\pi^{-}(addT)$ is a cluster tilting object in $D^b(\mathcal{H})$.

\item For any tilting object $T'$ in $\mathcal{H}$, $\oplus _{i=0}^{i=m-1}F^iT'$ is a cluster
tilting object in $\mathcal{C}_{F^m}(\mathcal{H})$, and any cluster
tilting object in $\mathcal{C}_{F^m}(\mathcal{H})$ arises in this
way, i.e. there is a hereditary abelian category $\mathcal{H}'$,
which is derived equivalent to $\mathcal{H}$, and a tilting object
$T$ in $\mathcal{H}'$ such that the cluster tilting object is
induced from $T$.

\end{enumerate}

\end{thm}

\begin{proof}

\begin{enumerate}
\item It follows Lemma 4.14 in [KZ] or the special case of Proposition 4.4 where $m=1$, that $T$ is a cluster
tilting object in $\mathcal{C}(\mathcal{H})$ if and only if $\pi
^{-1}(addT)$ is a cluster tilting subcategory in $D^b(\mathcal{H})$.
By Proposition 4.4,  we have that $\rho_m^{-1}(T)$ is a cluster
tilting object in $\mathcal{C}_{F^m}(\mathcal{H})$ if and only if
$\pi_m ^{-1}(add(\rho_m^{-1}(T)))$ is a cluster tilting subcategory
in $D^b(\mathcal{H})$. Since $\pi=\rho_m\pi_m$, $\pi (\pi_m
^{-1}(add(\rho_m^{-1}(T)))=T$,
 we have that $\rho_m^{-1}(T)$ is a cluster tilting
object in $\mathcal{C}_{F^m}(\mathcal{H})$ if and only if $T$ is a
cluster tilting object in $\mathcal{C(H)}$.

\item For any tilting object $T'$ in $\mathcal{H}$, from [BMRRT] and
[Zh],  $T'$ is a cluster tilting object in $\mathcal{C(H)}$. Hence
$\oplus _{i=0}^{i=m-1}F^iT'$ is a cluster tilting object in
$\mathcal{C}_{F^m}(\mathcal{H})$ by the first part of the theorem.
  Suppose $M$ is a cluster tilting object in
  $\mathcal{C}_{F^m}(\mathcal{H})$. Then by the first part of the
  theorem, $\rho_m(M)$ is a cluster tilting object in cluster
  category $\mathcal{C(H)}$. Therefore $\rho_m (M)$ is induced from a
  tilting object of a hereditary abelian category $\mathcal{H}'$,
  which is derived equivalent to $\mathcal{H}$ [Zh, BMRRT]. Then
  $M$ is induced from a tilting object of $\mathcal{H }'$.

\end{enumerate}
\end{proof}

\begin{defn} We call the endomorphism algebras
$End_{\mathcal{C}_{F^m}(\mathcal{H})}T$ of cluster tilting objects
$T$ in the generalized cluster category
$\mathcal{C}_{F^m}(\mathcal{H})$ the generalized cluster-tilted
algebras of type $\mathcal{H}$, or simply the generalized
cluster-tilted algebras.

\end{defn}

 Now we study the representation theory of generalized cluster-tilted algebras.
 We recall that $\pi_m: D^b(\mathcal{H})\longrightarrow
\mathcal{C}_{F^m}(\mathcal{H})$ is the projection.

\begin{thm}   Let $T$ be a tilting object in $\mathcal{H}$, $\widetilde{A}=
End_{\mathcal{C}_{F^m}(\mathcal{H})}(\oplus_{i=0}^{i=m-1}F^iT) $
 the generalized cluster-tilted algebra.
 \begin{enumerate}

\item  $\widetilde{A}$ has a Galois covering $\pi _m: A_{\infty}\rightarrow \widetilde{A}$ which is
 the restriction of the projection $\pi_m: D^b(\mathcal{H})\rightarrow \mathcal{C}_{F^m}(\mathcal{H})$.

\item The projection $\pi _m $ induces a push-down functor
 $\tilde{\pi} _m: \frac{D^b(\mathcal{H})}{add\{\tau^nT[-n]|n\in Z\}}
\longrightarrow \tilde{A}-mod$.

\item If $T'$ is a tilting object in $\mathcal{H}$, then the generalized cluster tilted algebra
 $\widetilde{A}'=End_{\mathcal{C}_{F^m}(\mathcal{H})}(\oplus_{i=0}^{i=m-1}F^iT')$
 has the same representation type as $A$.

 \end{enumerate}
\end{thm}

\begin{proof}
(1). Set $\mathcal{T}=add (\{\ F^i(T)\  |\  i\in \mathbf{Z}\  \})$.
$\mathcal{T}$ is a cluster tilting subcategory of
$D^b(\mathcal{H})$. Hence by Proposition 4.4 , $\pi_m (\mathcal{T})
$ is a cluster tilting object in $\mathcal{C}_{F^m}(\mathcal{H}).$
By Theorem 2.5, we have the equivalent functor
$Hom_{\mathcal{C}_{F^m}(\mathcal{H})}(\oplus_{i=0}^{i=m-1} \pi_m
(F^i(T)),-): \frac{\mathcal{C}_{F^m}(\mathcal{H})}{add(
\oplus_{i=0}^{i=m-1} \pi_m (F^i(T)))}\rightarrow \tilde{A}-mod$.
 Under this equivalence, the subcategory $add(\pi_m (\mathcal{T})) $
 correspondences to the subcategory of projective
 $\tilde{A}-$modules.

The projection $\pi _m$ sends $\mathcal{T}$ to $\pi_m(\mathcal{T})$.
Thus $(\pi_m)|_{\mathcal{T}}: \mathcal{T}\longrightarrow
\pi_m(\mathcal{T})$ is a Galois covering with Galois group generated
by $F^m$.

(2). By Theorem 3.3 and Corollary 4.4 in [KZ] there are equivalences
$D^b(H)/\mathcal{T}[1]\cong mod(\mathcal{T})$ and
$\mathcal{C}_m(\mathcal{H})/(\pi_m(\mathcal{T}[1]))\cong
mod(\pi_m(\mathcal{T}))$. We define the induced functor
$\bar{\pi}_m$ as follows: $\bar{\pi}_m (X):= \pi_m (X)$  for any
object $X\in D^b(H)/\mathcal{T}[1], $ and $\bar{\pi}_m
(\underline{f}):=\underline{\pi_{m} (f)}$ for any morphism
$\underline{f}: X\rightarrow Y$ in $D^b(H)/\mathcal{T}$. Clearly
$\bar{\pi}_m$ is well-defined and makes the following diagram
commutative:

\[ \begin{CD}
D^b(\mathcal{H})@>\pi_m>>\mathcal{C}_m(\mathcal{H})\\
@V P_1 VV@VV P_2 V \\
D^b(\mathcal{H})/\mathcal{T}[1]@>\bar{\pi}_m
>>\mathcal{C}_m(\mathcal{H})/\pi(\mathcal{T})[1].
    \end{CD} \]
Where $P_1, P_2$ are the natural quotient functors. Then
$\bar{\pi}_m$ is a covering functor from
$D^b(\mathcal{H})/\mathcal{T}[1]$ to
$\mathcal{C}_{F^m}/\pi(\mathcal{T}[1])$, i.e, it is a covering
functor from $D^b(\mathcal{H})/\mathcal{T}[1]$ to $\tilde{A}-mod$
 ($\approx mod (\pi_m(\mathcal{T}))$.

(3). This is direct consequence of Proposition 4.8 of [KZ].
\end{proof}

Similarly as above, the triangle covering functor $\rho_m :
\mathcal{C}_m(\mathcal{H})\rightarrow \mathcal{C(H)}$ induces a
covering functor from $\tilde{A}$ to the cluster-tilted algebra
$End_{\mathcal{C(H)}}T$ indicated as the following Theorem.

\begin{thm} Let $T$ be a tilting object in $\mathcal{H}$, $A=End_{\mathcal{C(H)}}T$ and $\widetilde{A}=
End_{\mathcal{C}_{F^m}(\mathcal{H})}(\oplus_{i=0}^{i=m-1}F^iT) $
 the generalized cluster-tilted algebra.
 \begin{enumerate}

\item   $\rho_m : \mathcal{C}_{F^m}(\mathcal{H})
\rightarrow \mathcal{C}(H)$ restricts to the cluster tilting
subcategory $add (\bigcup _{i=0}^{i=m-1}F^iT)$ induces a Galois
covering of $A$.

\item The functor $\rho _m $ also induces a push-down functor
 $\tilde{\rho _m}: \tilde{A}-mod \longrightarrow A-mod$.
\end{enumerate}

\end{thm}

\begin{proof} The strategy of the proof is almost the same as that of Theorem 4.7, we
present it here for the convenient of reader.

 (1). By Theorem 4.5, $\rho^{-1}_m(T)$ is a cluster tilting object in $\mathcal{C}_{F^m}(\mathcal{C})$, and
 $\rho^{-1}_m(T)=\oplus_{i=1}^{i=m-1}F^i(T)$.  By Theorem 2.5, we have the
equivalent functor
$Hom_{\mathcal{C}_{F^m}(\mathcal{H})}(\rho^{-1}_m(T),-):
\frac{\mathcal{C}_{F^m}(\mathcal{H})}{add(\rho^{-1}_m(T))}\rightarrow
\tilde{A}-mod$.
 Under this equivalence, the subcategory $add(\rho^{-1}_m(T)) $
 correspondences to the subcategory of projective
 $\tilde{A}-$modules.

The triangle functor $\rho _m$ sends add$\rho_m^-(T)$ to
 add$T$. Thus $\rho _m|_{add\rho_m^-(T)}:
add\rho_m^-(T)\longrightarrow addT$ is a Galois covering with Galois
group $Z_m$.

(2). By Theorem 3.3 and Corollary 4.4 in [KZ], there is an
 equivalence $\mathcal{C}_m(\mathcal{H})/(add\rho^{-1}_m(T)[1])
 \cong \tilde{A}-mod$. We define the induced functor $\bar{\rho}_m$
as follows: $\bar{\rho}_m(X):= \rho_m (X)$  for any object $X\in
\mathcal{C}_m(\mathcal{H})/(\rho_m^-(T))[1], $ and
$\bar{\rho}_m(\underline{f}):=\underline{\rho_{m} (f)}$ for any
morphism $\underline{f}: X\rightarrow Y$ in
$\mathcal{C}_m(\mathcal{H})/(\rho^{-1}_m(T))$. Clearly $\bar{\rho}
_m$ is well-defined and makes the following diagram commutative:

\[ \begin{CD}
\mathcal{C}_{F^m}(\mathcal{H}))@>\rho_m>>\mathcal{C}(\mathcal{H})\\
@VP'_1 VV@VV P'_2 V \\
\mathcal{C}_{F^m}(\mathcal{H})/add(\rho^{-1}_m(T))[1])@>\bar{\rho}
_m
>>\mathcal{C}(\mathcal{H})/add(T[1]).
    \end{CD} \]
Where $P'_1, P'_2$ are the natural quotient functors. Then
$\bar{\rho}_m$ is a covering functor from
$\mathcal{C}_m(\mathcal{H})/(add\rho_m^{-1}(T)[1])$ to
$\mathcal{C}/add(T[1])$, i.e, it is a covering functor from
$\tilde{A}-mod$ to $A-mod$.
\end{proof}

\begin{rem}
 By Theorem 3.1., see also [ABS,Zh], the cluster-tilted algebra $A$ of type $\mathcal{H}$ can be
 written as a trivial extension $A=B\ltimes M$, where $M=Ext^2_B(DB,B)$. Then $A$ has as
$\mathbf{Z}-$covering the following (infinite dimensional) matrix
algebra (i.e. the repetitive algebra):
  $$A_{\infty}=\left [\begin{array}{ccccc}\ddots&&&&\\
  \ddots&B&&& \\
  &M&B&&\\
  &&M&B&\\
  &&&\ddots&\ddots \end{array}\right ]$$

$A=B\ltimes M$ is also a $Z_m-$graded algebra. Then $A$ has a
$Z_m-$covering $A\sharp Z_m$, the smash product of graded algebra
$A$ with group $Z_m$.

\end{rem}

\textbf{Examples}

\begin{enumerate}

\item   Let $D^b(H)$ be the (bounded) derived category of hereditary
algebra $H$, where $H$ is the path algebra of the quiver
:\begin{center}
 \setlength{\unitlength}{0.61cm}
 \begin{picture}(5,4)
 \put(0,2){a}\put(0.4,2.2){$\circ$}
\put(3,2.2){$\circ$}\put(3.4,2){b}
\put(6,2.2){$\circ$}\put(6.4,2){c}

\put(0.8,2.5){\vector(3,0){2}}

\put(3.8,2.5){\vector(3,0){2}}

 \end{picture}
 \end{center}

 If we take $\mathcal{T}$ to
be the subcategory generated by $\{ \tau^{-n}P_a[n],
\tau^{-n}S_a[n], \tau^{-n}P_c[n]\mid n \in Z\},$ then $\mathcal{T}'$
is also a cluster tilting subcategory of $D^b(H)$ and
$D^b(H)/\mathcal{T}\cong A_{\infty}$ where $A_{\infty}$ is algebra
of quiver
$$A^{\infty}_{\infty}:   \cdots \circ \longrightarrow \circ
\longrightarrow \circ \longrightarrow\cdots $$ with $rad^2=0$ [KZ].

\item Let $m=1$. We consider the cluster category
 $\mathcal{C}(H)$. If we take $T=P_a\oplus P_c\oplus S_a$,  then $T$ is a cluster
tilting object of  $\mathcal{C}(H)$ and
 $\mathcal{C}(H)/(addT)\cong A$ where $A$ is an algebra
of quiver

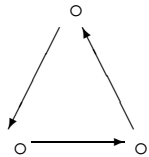
\begin{center}
 \setlength{\unitlength}{0.61cm}
 \begin{picture}(5,5)
 \put(0.4,2.2){$\circ$}
\put(3,2.2){$\circ$}
 \put(1.6, 5.2){$\circ$}

\put(0.8,2.5){\vector(3,0){2}}

\put(3,2.8){\vector(-1,2){1.1}}

 \put(1.4,5){\vector(-1,-2){1.1}}

 \end{picture}
 \end{center}

 with
$rad^2=0.$

\item Let $m=2$. We consider the generalized cluster category
 $\mathcal{C}_{F^2}(A)$. If we take $\mathcal{T}'$ to be the
subcategory generated by $\{ \tau^{-n}P_a[n], \tau^{-n}S_a[n],
\tau^{-n}P_c[n]\mid n=0, 1 \},$ then $\mathcal{T}'$ is a cluster
tilting subcategory of  $\mathcal{C}_{F^2}(A)$ and
 $\mathcal{C}_{F^2}(A)/\mathcal{T}\cong A_{1}$ where $A_{1}$ is an algebra
of quiver
$$Q_1:   \begin{array}{ccccc} \circ& \longrightarrow& \circ&
\longrightarrow& \circ\\
\uparrow&&&&\downarrow\\

\circ& \longleftarrow& \circ& \longleftarrow& \circ
\end{array}$$

 with
$rad^2=0.$

\item Let $m=3$. We consider the generalized cluster category
 $\mathcal{C}_{F^3}(A)$. If we take $\mathcal{T}'$ to be the
subcategory generated by $\{ \tau^{-n}P_a[n], \tau^{-n}S_a[n],
\tau^{-n}P_c[n]\mid n=0, 1 ,2 \},$ then $\mathcal{T}'$ is a cluster
tilting subcategory of  $\mathcal{C}_{F^3}(A)$ and
 $\mathcal{C}_{F^3}(A)/\mathcal{T}\cong A_{2}$ where $A_{2}$ is an algebra
of quiver
$$Q_2:   \begin{array}{ccccccccc} \circ& \longrightarrow& \circ&
\longrightarrow& \circ& \longrightarrow& \circ&\longrightarrow& \circ\\
\uparrow&&&&&&& \swarrow &\\
\circ& \longleftarrow& \circ& \longleftarrow& \circ&
\longleftarrow&\circ&
\end{array}$$

 with
$rad^2=0.$

\end{enumerate}



\begin{thebibliography}{99}


\bibitem[ABS1]{abs}
I.Assem, T.~Br\"ustle and R.~Schiffler.
\newblock Cluster-tilted algebras as trivial extensions.
\newblock Bull. London Math. \textbf{40}, 151-162, 2006.



\bibitem[ABS2]{abs}
I.Assem, T.~Br\"ustle and R.~Schiffler.
\newblock On the Galois coverings of a cluster-tilted algebra.
\newblock Preprint, {\tt arXiv:math.RT/0709.0805,} 2007.

\bibitem[BM]{BM}
A.Buan and R.Marsh.
\newblock Cluster-tilting theory.
\newblock Trends in representation theory of algebras and related topics,
Edited by J.de la Pe$\tilde{n}$a and R. Bautista. Contemporary
Mathematics.
 \textbf{406}, 1-30, 2006.


\bibitem[BMR1]{BMR}
A.Buan, R.Marsh, and I.Reiten.
\newblock Cluster-tilted algebras.
\newblock Transactions of the AMS \textbf{359}, 323-332, 2007.


\bibitem[BMR2]{BMR2}
A.Buan, R.Marsh, and I.Reiten.
\newblock Cluster mutation via quiver representations.
\newblock Comment. Math. Helv. \textbf{83} no. 1, 143-177, 2008.

\bibitem[BMRRT]{BMRRT}
A.Buan, R.Marsh, M.Reineke, I.Reiten and G.Todorov.
\newblock Tilting theory and cluster combinatorics.
\newblock Advances in Math. \textbf{204}, 572-618, 2006.

\bibitem[CCS1]{CCS1}
P.Caldero, F.Chapoton and R.Schiffler.
\newblock Quivers with relations arising from clusters ($A_n$ case).
\newblock Transactions
of the AMS. \textbf{358}, 1347-1364, 2006.


\bibitem[CK1]{CK}
P.Caldero and B.Keller.
\newblock From triangulated categories to cluster algebras. Invent.Math., \textbf{172}, 169-211, 2008.
\newblock


\bibitem[CK2]{CK}
P.Caldero and B.Keller.
\newblock From triangulated categories to cluster algebras II. II, Ann. Sci.¡äEcole
Norm. Sup. (4) 39 (2006), no. 6, 983¨C1009.
\newblock

\bibitem[FZ1]{FZ1}
S.Fomin and A.Zelevinsky.
\newblock Cluster Algebras I: Foundations.
\newblock J. Amer. Math. Soc. \textbf{15}, no. 2, 497--529, 2002.

\bibitem[GL]{GL}W. Geigle and H. Lenzing, Perpendicular categories with application to representations and sheaves.
J. Algebra, \textbf{144}, 273-343, 1991.


\bibitem[H]{H}
D.Happel.
\newblock Triangulated categories in the representation theory of quivers.
\newblock LMS Lecture Note Series, 119. Cambridge, 1988.



\bibitem[H2]{}
D.~Happel.
\newblock A characterization of hereditary categories with tilting object.
\newblock Invent. Math. \textbf{144}, 381-398, 2001.

\bibitem[I1]{I1}
O.Iyama. Higher dimensional Auslander-Reiten theory on maximal
orthogonal subcategories.Advances in Math. Vol.210(1), 22-50, 2007.

\bibitem[I2]{I2}
O.Iyama. Auslander correspondence. Advances in Math. Vol. 210(1),
51-82, 2007.

\bibitem[Ke]{Ke}
B.Keller.
\newblock Triangulated orbit categories.
\newblock Documenta Math. \textbf{10}, 551-581, 2005.

\bibitem[KR1]{KR1}
B.Keller, and I.Reiten.
\newblock Cluster-tilted algebras are Gorenstein and stably Calabi-Yau.
\newblock Adv.
Math. 211, 123-151, 2007.



\bibitem[KR2]{KR2}
B.Keller and I.Reiten.
\newblock Acyclic Calabi-Yau categories, with an appendix by Van.Den Bergh. Arxiv preprint math.RT/0610594, 2006.
 To appear in Compos. Math.
\bibitem[KZ]{KZ}

S.Koenig and B.Zhu.
\newblock From triangulated categories to abelian categories: cluster tilting in a general framework.
\newblock Math. Zeit. 258, 143-160, 2008.


\bibitem[MRZ]{MRZ}
R. Marsh, M. Reineke and A.Zelevinsky. Generalized associahedra via
quiver representations.
 \newblock Trans. AMS. \textbf{355}, 4171-4186, 2003.

\bibitem[Rin1]{Rin}
C.~M.~Ringel.
\newblock Tame algebras and integral quadratic forms.
\newblock Lecture Notes in Mathematics, 1099. Springer-Verlag,
Berlin, 1984.

\bibitem[Rin2]{Rin}
C.~M.~Ringel. Some remarks concerning tilting modules and tilted
algebras. Origin. Relevance. Future. An appendix to the Handbook of
tilting theory, edited by L. Angeleri-H\"ugel, D.Happel and H Krause.
Cambridge University Press, LMS Lecture Notes Series 332, 2007.



\bibitem[Zh]{Z1}
B.Zhu.
\newblock Equivalences between cluster categories,
\newblock J. Algebra \textbf{304}, 832-850, 2006.


\end{thebibliography}
\end{document}